\documentclass[12pt,english]{smfart}
\usepackage{newtxtext,mathptmx,dsfont,MnSymbol}

\usepackage[text={6.75in, 9in},centering]{geometry}

\usepackage[all]{xy}
\usepackage[dvipsnames]{xcolor}   
\usepackage{xparse}
\usepackage{xr-hyper}
\usepackage[linktocpage=true,colorlinks=true,hyperindex,citecolor=teal,linkcolor=blue]{hyperref}  

\DeclareMathOperator{\m}{\mu}
\DeclareMathOperator{\lm}{\leadsto_{\!\!\m}}

\newcommand{\inv}{^{\raisebox{.2ex}{$\scriptscriptstyle-1$}}}

\newtheorem{theorem}{Theorem}[section]
\newtheorem{proposition}[theorem]{Proposition}

\newtheorem{corollary}[theorem]{Corollary}
\theoremstyle{definition}
\newtheorem{definition}{Definition}[section]
\newtheorem{remark}[theorem]{Remark}
\newtheorem{example}[theorem]{Example}
\numberwithin{equation}{section}

\renewcommand{\labelenumi}{\upshape(\arabic{enumi})}

\def\leq{\leqslant}

\begin{document}
	
\title{$\m$-elements: An extension of essential elements}
	
\author{Elena Caviglia}

\address{[1] Department of Mathematical Sciences, Stellenbosch University, South Africa. [2]  National Institute for Theoretical and Computational Sciences (NITheCS), South Africa.}
\email{elena.caviglia@outlook.com}
	
\author{Amartya Goswami}
\address{[1]  Department of Mathematics and Applied Mathematics, University of Johannesburg, P.O. Box 524, Auckland Park 2006, South Africa. [2]  National Institute for Theoretical and Computational Sciences (NITheCS), South Africa.}
\email{agoswami@uj.ac.za}
	
\author{Luca Mesiti}
\address{Department of Mathematics, University of KwaZulu-Natal, South Africa.}
\email{luca.mesiti@outlook.com}
	
%quantales;
	
%Ideals and multiplicative ideal theory in commutative rings
	
%Complete lattices, completions
	
	
\begin{abstract}
We introduce and study $\m$-elements, that generalize a lattice-theoretic abstraction (namely, essential elements) of essential ideals of rings, essential submodules of modules, and dense subsets of topological spaces. Exploring several examples, we show that $\m$-elements are indeed a genuine extension of essential elements.  We study preservation of $\m$-elements under contractions and extensions of quantale homomorphisms. We introduce $\m$-complements and $\m$-closedness and study their properties. We  determine $\m$-elements for several distinguished quantales,  including ideals of $\mathds{Z}_n$ and open subsets of topological spaces. Finally, we provide a complete characterization of $\m$-elements in modular quantales.
\end{abstract}
	
\maketitle 
	
\section{Introduction}
This paper introduces and studies various aspects of $\m$-elements, in the setting of quantales. The concept of 
$\m$-elements provides a natural generalization of essential elements, shifting from considering meets with only one element to consider meets with any finite number of elements. Recall that an element $x$ in a lattice $L$ is called \emph{essential}\footnote{also known as  \emph{dense} element (see \cite{DB23}) or \emph{large} element (see \cite{Row88}).} if $x\wedge y\neq 0$ for all $y\neq 0$ in $L$. It is clear that essential elements can be viewed as a lattice-theoretic abstraction of both dense subsets of topological spaces and essential submodules (and in particular, essential ideals) of rings. Indeed, the set of open subsets of a topological space forms a frame, while the set of submodules of a module forms a modular lattice. It is well-known (see \cite{Row88}) that the socle of a module is nothing but the intersection of essential submodules of it, and a module $E$ is called the injective hull of a module $M$, if $E$ is an essential extension of $M$, and $E$ is injective. It is no surprise that a similar result for the socle of a modular lattice holds for essential elements, and that there is a notion of essential extension with respect to essential elements (see \cite{Cal00}).  Moreover, comprehensive studies of essential ideals of rings have also been done (see \cite{OJ81, Tah15, Aza97, Aza95, Mom10, Gre05}). On the other hand, dense elements of lattices shed light on the lattice-theoretic aspects of dense subsets of topological spaces, and these have been well studied (see \cite{DB23, BD22, HM07, KRB22, EN75}).

Our notion of $\m$-elements abstracts the concept of $\mathcal{M}$-ideals (more generally, $\mathcal{M}$-submodules) as defined in \cite{BG25}, which themselves form a ring-theoretic extension of $\mathcal{M}$-ideals of Banach spaces introduced in \cite{AE72}. Let us briefly see how the transition from Banach spaces to rings occurs. Suppose that $V$ is a real Banach space, and that $W$ is the dual Banach
space of $V$. A subspace $N$ of $W$ is said to be an $L$-summand of $W$ if there is
a subspace $M$ with $N\cap M=\{0\}$, $N+M=W$, and for $p\in N$, $q\in M$, $\| p+q\|=\|p\|+\|q\|$.  A closed subspace $J$ of $V$ is said to be an $\mathcal{M}$-ideal if its annihilator $J^\perp$
is an $L$-summand in $W$.  According to \cite[Theorem 5.8]{AE72}, an ideal $J$ is an $\mathcal{M}$-ideal if and only if the following condition holds: if $B_1$, $\ldots$, $B_n$ are open balls with $B_1\cap \cdots\cap B_n\neq \emptyset$ and $B_i\cap J\neq \emptyset$ for all $i$, then $B_1\cap \cdots\cap B_n \cap J\neq \emptyset$. Now, note that we can replace the Banach space by a ring, and $J$ and all open balls by ideals, so that we obtain the definition of a ``ring $\mathcal{M}$-ideal'' as follows. An ideal \(J \) of a ring $R$ is called an \emph{\( \mathcal{M} \)-ideal} if for a finite collection of nonzero ideals $I_1$, $I_2$, $\dots$, $I_n $ of $R$ with  $\bigcap_{k=1}^n I_k\neq 0$, the condition \(J \cap I_k \neq 0\) for each \(k\), where $k\in \{1, \ldots, n\}$, implies that
\[	J \cap \left( \bigcap_{k=1}^n I_k \right) \neq 0.
\]
And we can then abstract this concept to the context of lattices (see Definition \ref{mid}). Interestingly enough, $\mathcal{M}$-ideals of rings turn out to be a generalization of essential ideals, and hence, our study is well-motivated. One major ``structural'' difference between $\mathcal{M}$-ideals of rings and $\mathcal{M}$-ideals of Banach spaces is the following. For $\mathcal{M}$-ideals of Banach spaces, it is necessary and sufficient to consider only three open balls (see \cite[Theorem 5.9]{AE72}), whereas for $\mathcal{M}$-ideals of rings (and hence, $\m$-elements of quantales), in our definition above, two ideals are sufficient (see Proposition \ref{net} for $\m$-elements version).

The goal of this paper is twofold. Firstly, we concentrate on studying several interesting properties that $\m$-elements enjoy in a quantale setting. Secondly, we show in what respects they behave differently from essential elements.

After briefly recalling preliminary definitions in Section \ref{bck}, we introduce $\m$-elements in Section \ref{muel}. Through several examples, we establish that $\m$-elements are a natural and genuine generalization of essential elements.  We study preservation of $\m$-elements under contractions and extensions of quantale homomorphisms. Several properties of $\m$-elements are proved specifically for upsets. The purpose of Section \ref{mucc} is to introduce $\m$-complements and $\m$-closedness, that generalize essential complements and essential closedness, and study their properties. One of the main results in this section is to prove the existence of all $\m$-complements in every modular quantale. In Section \ref{chmu}, we give characterizations of $\m$-elements in several distinguished quantales, which include quantales of ideals of $\mathds{Z}_n$ and more in general of quotients of PIDs, and frames of open sets of topological spaces. Finally, we present a complete characterization of $\m$-elements in modular quantales: we prove that every $\m$-element is either essential or irreducible. 
	
\section{Background}
\label{bck}

In this preparatory section, we briefly recall some basic concepts, and we fix the notation. An  \emph{integral commutative quantale} is a complete lattice $(L, \leqslant, 0, 1)$ endowed with an associative, commutative multiplication (denoted by $\cdot$), which distributes over arbitrary joins and has $1$ as multiplicative identity. This definition was first introduced in \cite{WD39} (also see \cite{Dil62}), where the authors referred to it as a \emph{multiplicative lattice}. Throughout this paper, unless otherwise stated, by a ``quantale,'' we shall always refer to an integral commutative quantale.
For brevity, we shall write $xy$ for $x\cdot y$ and  $x^n$ for $x\cdot   \cdots \cdot   x$ (repeated $n$ times). By $x<y$, we shall mean $x\leqslant y$ and $x\neq y$. We say a quantale is \emph{distributive} (respectively, \emph{modular}) if the underlying lattice is distributive (respectively, modular). We say that a map $\varphi\colon L \to L'$ is a \emph{quantale homomorphism} if $\varphi$ preserves arbitrary joins, multiplication, and $1$.
	
%For our later purposes, in the following lemma, we collect some elementary properties of multiplication in quantales.
	
%\begin{lemma}\label{bip} In a quantale $L$, the following hold. \begin{enumerate}
			
%\item\label{pxyx} $xy\leqslant x$, for all $x$, $y\in L.$
			
%\item \label{mul} $x  y\leqslant x\wedge y$, for all $x, y\in L$.
			
%\item $x  0=0$, for all $x\in L$.
			
%\item\label{mon} If $x\leqslant y$, then $x  z\leqslant y  z,$\, for all $x$, $y$, $z\in L$.
			
%\item\label{monj} If $x\leqslant y$ and $u\leqslant v$, then $x  u\leqslant y v$, for all $x,$ $y,$ $u,$ $v\in L$.
%\end{enumerate}
%\end{lemma} 
	
%\begin{proof} For (1), 	we notice that \[x=x  1=x  (y\vee 1)=(x y)\vee x,\] and that implies $x  y\leqslant x$. The claim (2)  follows from \ref{pxyx}. For	 (3), we  apply (\ref{mul}) and obtain $x  0\leqslant x\wedge 0=0.$ Hence $x  0=0.$ Since $x\leqslant y$, we have \[y  z=(x\vee y)  z=(x  z)\vee (y  z),\] which implies that $x  z\leqslant y  z$, and this proves (4). Lastly, to have (5), applying (\ref{mon}) on $x\leqslant y$ and $u\leqslant v$ give us $x  u\leqslant y  u$ and $y  u\leqslant y  v$, and hence $xu\leqslant yv$. 
	%\end{proof} 
	
Let us recall a few definitions that will be needed in this paper. A \emph{frame} is a quantale, where meet coincides with multiplication.  The \emph{upset} and  \emph{downset} of an element $a$ in $L$ are respectively defined as $a^{\uparrow}:=\{x\in L \mid a\leqslant x\leqslant 1\}$ and $a^\downarrow:=\{ x\in L\mid 0\leqslant x\leqslant a\}$. A \emph{proper} element $x$ in $L$ is such that $x\neq 1$. A \emph{complement} for an element $x$ in a quantale $L$ is an element $y\in L$ such that $x\wedge y=0$ and $x\vee y=1$. A \emph{pseudo-complement} for an element $x$ in a quantale $L$ is an element $y\in L$ that is the greatest element for which $x\wedge y=0$. We say that a quantale $L$ is \emph{complemented} (respectively, \emph{pseudo-complemented}) if every  element $x\in L$ has a complement (respectively, pseudo-complement) in $L$. A subset $\{a_i\}_{i\in I}$ of a quantale $L$ is called \emph{independent} if for every $i$, the following equality holds:\[a_i\wedge \left(  \bigvee_{i\neq j, j\in I} a_j\right)=0.\]  
Given an element $a$ in $L$, an element $b$ in $L$ is called \emph{minimal above} $a$ if $a<b$ and for any element $c\in L$ with the property that $a<c\leqslant b$ implies that $c=b.$ An \emph{atom} is a minimal element above $0$. An element $x$ (with $x\neq 1$) is said to be \emph{maximal} if $x<y\leqslant 1$, then $y=1$. We say a quantale is \emph{local} if it has exactly one maximal element. %Following \cite{PP16}, we say a quantale $L$ is \emph{max-bounded} if every proper element of $L$ is below of a maximal element in $L$.  An element $a$ in $L$ is called \emph{nilpotent} if we have $a^n=0$ for some $n\in \mathds{N}^+$, where $\mathds{N}^+$ is the set of positive natural numbers.
%, and $L$ is said to be \emph{reduced} if $0$ is the only nilpotent element in $L$. 
An element $c$ in $L$ is called \emph{compact} if whenever  $\{x_{\lambda}\}_{\lambda \in \Lambda}\subseteq L$ and $c\leqslant \bigvee_{\lambda \in \Lambda} x_{\lambda}$, we have $c\leqslant \bigvee_{i=1}^n x_{\lambda_i}$, for some $n\in \mathds{N}^+$(where $\mathds{N}^+$ is the set of positive natural numbers), and  $L$ is said to be \emph{compactly generated} if every element of
$L$ is the join of compact elements in $L$. 
The \emph{annihilator} of an element $a$ in $L$ is defined as \[a^{\perp}:=\bigvee\{ x\in L\mid xa=0\}.\]
To denote the annihilator of the annihilator $a^{\perp}$  of an element $a$, instead of $(a^{\perp})^{\perp}$, we shall use the notation $a^{\perp\perp}$. An element $a\in L$ is said to be \emph{regular} if  $a^{\perp\perp}=a$. Recall that the join of all atoms of a quantale $L$ is called \emph{socle}, and we denote it by $\mathrm{Soc}(L)$. 

A comprehensive discussion regarding lattices, frames, and quantales can be found in \cite{Cal00}, \cite{PP12}, and \cite{Ros90}, respectively.
	
\section{$\m$-elements}\label{muel}
	
In this section, we present and explore the notion of $\m$-elements, which naturally extends the notion of essential elements. Interestingly, both essential elements and atoms are always $\m$-elements. We generalize many properties and theorems of essential elements to $\m$-elements. Furthermore, we show that while some properties of essential elements extend to 
$\m$-elements, others do not.
		
\begin{definition}
\label{mid}
An element \(x \) in a quantale $L$ is called a \emph{$\m$-element} if for every natural number $n\geqslant 2$ and every collection of $n$ elements $y_1$, $ y_2$, $\ldots$, $y_n $ in $L$ such that $\bigwedge_{k=1}^n y_k\neq 0$, the condition \(x \wedge y_k \neq 0\) for each \(k\),  implies that
\[	x \wedge \left( \bigwedge_{k=1}^n y_k \right) \neq 0.
\]
We write $x\lm b^\downarrow$ (resp. $x\lm b^\uparrow$) to mean that $x\leqslant b$ (resp. $b\leqslant x$) and $x$ is a $\m$-element in $b^{\downarrow}$ (resp. in $b^{\uparrow}$). 
\end{definition}
	
The notion of a
$\m$-element generalizes that of an essential element by considering all finite collections of elements $y_1$, $y_2$, \ldots, $y_n$ (with $n\geqslant 2$) rather than just a single element $y$. We shall see that essential elements are always 
$\m$-elements and that there exist 
$\m$-elements that are not essential elements.
	
For the rest of the paper, we shall work with quantales unless explicitly stated otherwise. This will help, once and for all, in avoiding various assumptions such as bounded lattices, complete lattices, upper continuous lattices, \textit{etc.}, which would be required in specific contexts. However, when we deal with $a^\uparrow$, we need to restrict ourselves to frames. This is simply because  $a^\uparrow$ is not, in general, a sub-quantale of a quantale $L$. One  may ask, why not then work with frames only? The answer is the set of all ideals of a commutative ring  with identity  is a quantale, but not a frame in general; and ideals are going to be our vital source of examples (and counterexamples) of $\m$-elements.

Note that in the literature, whenever $a$ is an essential element in $b^\downarrow$, $b$ is called an \emph{essential extension of} $a$, and it is also denoted by $a\leqslant_e b$. However, there is no similar terminology (and hence notation) for $a$ being essential in $b^\uparrow$.  We have proposed the notation $a\lm\star$ to accommodate both $b^\downarrow$ and $b^\uparrow$; otherwise, a notation like $a\leqslant_{\mu} b^\uparrow$ would be strange, as in this case $a\nleqslant b$.

\begin{proposition}\label{eim}
Every essential element in a quantale is a \( \m \)-element.
\end{proposition}
	
\begin{proof}
Let $x\in L$ be an essential element. Let then $n\geqslant 2$ and $y_1$, $ y_2$, $\ldots$, $y_n $ in $L$ such that $\bigwedge_{k=1}^n y_k\neq 0$ and \(x \wedge y_k \neq 0\) for each \(k\). Since $x$ is essential, the condition $\bigwedge_{k=1}^n y_k\neq 0$ implies that $x\wedge \left(\bigwedge_{k=1}^n y_k\right)\neq 0$. So $x$ is a $\m$-element.
\end{proof}
	
\begin{remark}\label{remuselessincheck}
To check whether an element $x$ of a quantale is a $\m$-element or not, it is unnecessary to consider $1$ among the elements $y_1,\ldots ,y_n$. In fact, it is always useless to add elements $\overline{y}$ that are larger than some $y_i$ that are already being considered. Adding $x$ itself is also unnecessary.
\end{remark}
	
\begin{remark}\label{notalless}
We may ask whether every 
$\m$-element is an essential element. The answer is negative. To see this, let us consider the quantale of all ideals of the ring $\mathds{Z}_{12}$. Let $x:= 3\mathds{Z}_{12}$. The only proper ideals different from $x$ that intersect $x$ are $2 \mathds{Z}_{12}$ and $6 \mathds{Z}_{12}$ (note that $3\mathds{Z}_{12}\wedge 4\mathds{Z}_{12}$=0). So by Remark \ref{remuselessincheck}, we only need to consider $y_1:= 2\mathds{Z}_{12}$ and $y_2:= 6\mathds{Z}_{12}$. It is easy to see that $x\wedge y_1\wedge y_2=6 \mathds{Z}_{12}\neq 0$. So we conclude that $3\mathds{Z}_{12}$ is a $\m$-element. However, $ 3\mathds{Z}_{12}$ is not an essential element because $ 3\mathds{Z}_{12} \wedge 4\mathds{Z}_{12}=0.$
		
Another example is given by the frame $\mathcal{P}(X)$ of the power set of a nonempty set $X$. It is easy to see that every singleton set in $X$ is a $\m$-element. This is not an accidental example, as we shall see later on in Proposition \ref{bpmi}(\ref{atmu}). And when $X$ has more than one element, singleton subsets of $X$ are not essential.
\end{remark}
	
\begin{remark}\label{notallmuel}
It is not true that all elements in a quantale are $\m$-elements. Consider the quantale of all ideals of the ring $\mathds{Z}_{30}$ and the ideal $x:= 2\mathds{Z}_{30}$. If we take $y_1:= 3\mathds{Z}_{30} $ and $y_2:=  5\mathds{Z}_{30} $, then $x\wedge y_1= 6\mathds{Z}_{30}$, $x\wedge y_2= 10\mathds{Z}_{30}$, and $y_1\wedge y_2= 15\mathds{Z}_{30}$. However, $x\wedge (y_1\wedge y_2)=0$. 
		
Another example is given by the frame of power set $\mathcal{P}(X)$ of the set $X=\{1,2,3\}$.The set $x=\{1,2\}$ is not a $\m$-element in $\mathcal{P}(X)$. Indeed, the sets $y=\{1,3\}$ and $z=\{2,3\}$ are such that $y \wedge z =\{3\} \neq \emptyset$, $x \wedge y=\{1\} \neq \emptyset$ and $x\wedge z=\{2\} \neq \emptyset$, but $x \wedge y \wedge z=\emptyset$.
\end{remark}
	
\begin{remark}
Proposition \ref{eim}, Remark \ref{notalless}, and Remark \ref{notallmuel} show that $\m$-elements genuinely extend the notion of essential elements.
\end{remark}
	
We have seen that $\m$-elements generalize the notion of essential elements by considering all finite collections of elements $y_1$, $y_2$, \ldots, $y_n$ rather than just a single element $y$. A natural question arises: can this condition be reduced to just $n=2$? The next proposition confirms that this is indeed possible.
	
\begin{proposition}
	\label{net}
Let $L$ be a quantale and let $a\in L$. The following are equivalent:
\begin{enumerate}
\item $a$ is a $\m$-element;

\item  for every $y$, $z\in L$ such that $y \wedge z\neq 0$, the conditions $a\wedge y \neq 0$ and $a\wedge z\neq 0$ imply that ${a\wedge (y \wedge z) \neq 0}$.	\end{enumerate}
\end{proposition}
	
\begin{proof}
(1) trivially implies (2). We prove (2) implies (1) by induction. Observe that by (2), the property holds for $n=2$. Assume that the property holds for $n$. Suppose $\{y_1,\ldots,y_{n+1}\}$ is a set of $n+1$ elements in $L$ such that $ \bigwedge_{k=1}^{n+1} y_k \neq 0$ and for each $k$ we have $a\wedge y_k\neq 0$. In particular, we  also have $ \bigwedge_{k=1}^{n} y_k \neq 0$, so by inductive hypothesis, we obtain $a\wedge (\bigwedge_{k=1}^{n} y_k ) \neq 0$. Since $(\bigwedge_{k=1}^{n} y_k) \wedge y_{n+1}= \bigwedge_{k=1}^{n+1} y_k \neq 0$ and $a \wedge y_{n+1} \neq 0$, by (2) we conclude that 
\[a \wedge \left(\bigwedge_{k=1}^{n+1} y_k\right)= a \wedge \left(\bigwedge_{k=1}^{n} y_k\right) \wedge y_{n+1}\neq 0.\qedhere\]
\end{proof}
	
Thanks to the above result, we shall henceforth verify the condition of Definition \ref{mid} only for 
$n=2$ in all our proofs.
	
\begin{proposition}\label{bpmi}
In a quantale $L$,  the following hold.
\begin{enumerate}
\item The bottom element is  a $\m$-element.
			
\item\label{temu} For every $a\in L$, we have $a\lm a^\downarrow$. Hence, the top element in $L$ is a $\m$-element.
			
\item\label{aid} Every element in a domain is a $\m$-element.
			
\item The maximal element in a nontrivial local $L$ with compact top element is a \( \m \)-element.
			
\item\label{atmu} Every atom in $L$ is a $\m$-element.
			
			%\item If \( x\) is an \( \m \)-element of  \( L \), then \( \mathrm{Rad}(x) \) is also a \( \m \)-element. \add{???}
			
			%\item A nonzero nilradical \(\mathrm{Nil}(L)\) of $L$  is a \(\m\)-element.  \add{???}
			
			%\item A nonzero Jacobson radical $\mathrm{Jac}(L)$ of \(L\)   is a \(\m\)-element. \add{???}
\end{enumerate}
\end{proposition}
	
\begin{proof}
(1) is obvious, as the condition of being a $\m$-element is empty for the bottom element.
		
The claims (2) and (3) are immediate consequences of Proposition \ref{eim}, on the other hand, (4) follows from \cite[Lemma 2.3(2)]{GD25} and Proposition \ref{eim}. 
To show (5), suppose $a$ is an atom in $L$. Let $x_1$ and $x_2$ be nonzero elements in $L$ such that $x_1\wedge x_2\neq 0$ and $a\wedge x_k\neq 0$ for each $k$. Since $a$ is an atom, this implies that $a\wedge x_k =a$, for each $k$. Hence $a\wedge \left( x_1\wedge x_2 \right)=a\neq0$, proving that $a$ is a $\m$-element.
\end{proof}
	
\begin{remark}
Although 
$\m$-elements are a generalization of essential elements, they interestingly include all atoms, which are, in general, very far from being essential.
\end{remark}
	
\begin{proposition}
Let $L$ be a quantale. The meet of two \( \m \)-elements in \( L \)
is a \( \m \)-element.
\end{proposition}

\begin{proof}
Suppose $a$ and $b$ are $\m$-elements in $L$. Let $x_1$ and $x_2$ be nonzero elements in $L$ such that $x_1\wedge x_2\neq 0$ and $(a\wedge b)\wedge x_k\neq 0$ for each $k$. This, in particular, implies that  $b\wedge x_k\neq 0$ for each $k$. Since $b$ is a $\m$-element, we have $b\wedge \left(  x_1\wedge x_2\right)\neq 0$. Now consider the nonzero elements $b\wedge x_1$ and $b\wedge x_2$ in $L$. Observe that \[\bigwedge_{k=1}^2(b\wedge x_k)=b\wedge \left(  x_1\wedge x_2\right)\neq 0,\]
and $a\wedge(b\wedge x_k)\neq 0$ for each $k$. Since $a$ is a $\m$-element, this implies that
\[(a\wedge b)\wedge \left(  x_1\wedge x_2\right)=a\wedge \left(  \bigwedge_{k=1}^2 (b\wedge x_k)\right)\neq 0,\]
proving $a\wedge b$ is a $\m$-element.
\end{proof}
	
\begin{remark}\label{remoperations}
Similar to essential elements, here are instances where $\m$-elements fail to be closed under certain operations.	\begin{itemize}
\item[$\bullet$] An arbitrary meet of $\m$-elements may not be a $\m$-element. For example, in the quantale of all ideals of the ring $\mathds{Z}$ of integers, for each prime $p$, the ideal $p\mathds{Z}$ is a $\m$-element by Proposition \ref{bpmi}(\ref{aid}).   However, $\bigwedge_{p}p\mathds{Z}=0$.
			
\item[$\bullet$]  Join of two $\m$-elements may also not be a $\m$-element. Consider two nonempty subsets $A$ and $B$ in the quantale of power set $\mathcal{P}(X)$ of a set $X$. Assume that $A\wedge B\neq \emptyset$ and that $A$ and $B$ are not contained in the other. Let $x\in A$ but $x\notin B$, and let $y\in B$ but $y\notin A$. By Proposition \ref{bpmi}(\ref{atmu}), each $\{x\}$ and $\{y\}$ is a $\m$-element. However, $\{x,y\}$ is not a $\m$-element. Indeed, $A\wedge B\neq \emptyset$, $\{x,y\}\wedge A\neq \emptyset$ and $\{x,y\}\wedge B\neq \emptyset$, but $\{x,y\}\wedge (A\wedge B)=\emptyset$.
			
\item[$\bullet$] Complements of $\m$-elements may not be $\m$-elements. In fact, even complements of atoms may not be $\m$-elements. Consider the quantale of power set $\mathcal{P}(X)$ of the set $X=\{1,2,3\}$. The singleton $\{1\}$ is an atom and thus a $\m$-element. But its complement $\{2,3\}$ is not a $\m$-element: consider $y=\{1,2\}$ and $z=\{1,3\}$. Notice that $\{2,3\}$ is also a dual atom which is not a $\m$-element.
			
\item[$\bullet$] The product of two $\m$-elements may not be a $\m$-element; see Example \ref{exprodmuel}.
\end{itemize}
\end{remark}
	
It is clear that if $a$ is a $\m$-element in $L$ and $a<b$, then $b$ may not be a $\m$-element. Indeed, consider again, as in Remark \ref{remoperations}, the quantale of power set $\mathcal{P}(X)$ of $X=\{1,2,3\}$. The singleton set $\{2\}$ is a $\m$-element, but $\{2,3\}$ is not.

However, we have the following positive result for upsets, and the proof is obvious.

\begin{proposition}
Let $L$ be a frame and $a$, $c$, $d\in L$. If $a\lm c^\uparrow$ and $c\leq d\leq a$, then $a\lm d^\uparrow$.
\end{proposition}

\begin{remark}\label{remrem}
Unlike what happens for essential elements (see \cite[Exercise 4.8]{Cal00}), it is not true that if $b\lm a^\uparrow$ (with $a\leq b$), then $b$ needs to be a $\m$-element in the frame $L$. Consider again the frame of power set $\mathcal{P}(X)$ of the set $X=\{1,2,3\}$. Then $\{2,3\}$ is an atom in $\{2\}^\uparrow$ and thus a $\m$-element in $\{2\}^\uparrow$. But $\{2,3\}$ is not a $\m$-element in $\mathcal{P}(X)$.
\end{remark}
	
\begin{proposition}
Let \( \varphi \colon L \to L' \) be an injective quantale homomorphism. If $y$ is a $\m$-element in $L'$, then \( \varphi\inv(y)\) is a $\m$-element in $L$.
\end{proposition} 
	
\begin{proof}
Suppose \( x := \varphi\inv(y)\). Let $x_1$ and $ x_2$  be nonzero elements in \( L \) such that 
$
x_1\wedge x_2 \neq 0$, and  for each $k$, suppose $x \wedge x_k \neq 0$.
Since $\varphi$ is injective, we have the following conclusions:
\begin{itemize}
\item[$\bullet$] $\varphi(x_k)\neq 0$ for each $k$; 
			
\item[$\bullet$] $0\neq \varphi(x_1\wedge x_2)\leqslant \varphi(x_k)\wedge \varphi(x_2)$;
			
\item[$\bullet$] $0\neq \varphi(x\wedge x_k)\leqslant \varphi(x)\wedge \varphi(x_k)=y\wedge \varphi(x_k)$ for each $k$.
\end{itemize}
Since $y$ is a $\m$-element in $L'$, we have  $y\wedge \left( \varphi(x_1)\wedge \varphi(x_2)  \right)\neq 0,$ and that gives
\[\varphi\inv(y)\wedge \bigwedge_{k=1}^2 (\varphi\inv(\varphi(x_k)))=x\wedge \left(  x_1\wedge x_2\right)\neq 0,\]
proving the claim.
\end{proof}  
	
In the last result, we have given a sufficient condition for contractions of $\m$-elements to be $\m$-elements. The next proposition is about characterization of certain $\m$-elements  extensions, and this result extends \cite[Lemma 4.10]{DB23}. 
	
\begin{proposition}
Let $L$ be a frame and $a\in L$. Consider the map $\kappa_a\colon L \to a^\uparrow$ defined by \[\kappa_a(x):=x\vee a.\] The following properties are equivalent:
\begin{enumerate}
\item $\kappa_a$ preserves essential elements.

\item $a$ is regular.

\item $\kappa_a$ preserves essential elements and $\m$-elements.
\end{enumerate}
\end{proposition}
	
\begin{proof}
It has been shown in \cite[Lemma 4.10]{DB23} that $\kappa_a$ preserves essential elements if and only if $a$ is regular. Therefore, it remains for us to show that if $a$ is regular then $\kappa_a$ also preserves $\mu$-elements. So, let $x\in L$ be a $\m$-element in $L$; we prove that $x\vee a$ is a $\m$-element in $a^\uparrow$. Let $y> a$ and $z> a$ such that $y\wedge z>a$, $(x\vee a)\wedge y> a$ and $(x\vee a)\wedge z> a$. Notice that $(x\vee a)\wedge y> a$ is equivalent to $(x\wedge y)\vee a > a$, which is equivalent to $x\wedge y \not\leqslant a$, and analogously for $z$ in place of $y$. We need to show that $(x\vee a)\wedge y\wedge z> a$, which is equivalent to $x\wedge y\wedge z\not\leqslant a$. By assumption $a$ is regular, so $a^{\perp\perp}\leq a$. It is easy to see that 
$$a^{\perp\perp}=\bigvee \{l\in L \;\;|\;\; \forall w\in L \;\;(w\wedge a=0 \,\Rightarrow\, w\wedge l=0)\}.$$
So, $a^{\perp\perp}\leq a$ means that, for every $l\in L$, the condition $\forall w\in L \;\;(w\wedge a=0 \,\Rightarrow\, w\wedge l=0)$ implies that $l\leq a$. And of course, we also have that if $l\leq a$, then $\forall w\in L \;\;(w\wedge a=0 \,\Rightarrow\, w\wedge l=0)$. The thesis $x\wedge y\wedge z\not\leqslant a$ is thus equivalent to finding some $w\in L$ such that $w\wedge a=0$, but $w\wedge x\wedge y\wedge z\neq 0$.
		
Since $y\wedge z \not\leqslant a$, there needs to exist some $w^{y,z}\in L$ such that $w^{y,z}\wedge a=0$, but $w^{y,z}\wedge y\wedge z\neq 0$. Analogously, since $x\wedge y\not\leqslant a$, there exists some $w^{x,y}\in L$ such that $w^{x,y}\wedge a=0$ but $w^{x,y}\wedge x\wedge y\neq 0$, and there also exists some analogous $w^{x,z}$. Now, either $x\wedge w^{y,z}\neq 0$ or $x\wedge w^{y,z}= 0$.
If $x\wedge w^{y,z}\neq 0$, consider the elements $w^{y,z},y,z$. We have that $x\wedge y\neq 0$ because otherwise $x\wedge y\leq a$, and analogously $x\wedge z\neq 0$. Since $w^{y,z}\wedge y\wedge z\neq 0$ and $x$ is a $\m$-element in $L$, we conclude that $x\wedge w^{y,z}\wedge y\wedge z\neq 0$. So $w=w^{y,z}$ works. 
If instead $x\wedge w^{y,z}= 0$, consider the elements $(w^{x,y}\wedge y)\vee w^{y,z}$ and $(w^{x,z}\wedge z)\vee w^{y,z}$. We have that,
$$((w^{x,y}\wedge y)\vee w^{y,z})\wedge ((w^{x,z}\wedge z)\vee w^{y,z})\geqslant w^{y,z} \wedge w^{y,z}=w^{y,z}\neq 0.$$
Moreover $x\wedge ((w^{x,y}\wedge y)\vee w^{y,z})=x\wedge w^{x,y}\wedge y \neq 0$, and analogously $x\wedge ((w^{x,z}\wedge z)\vee w^{y,z})=x\wedge w^{x,z}\wedge z\neq 0$. Since $x$ is a $\m$-element, we conclude that
$$x\wedge ((w^{x,y}\wedge y)\vee w^{y,z})\wedge ((w^{x,z}\wedge z)\vee w^{y,z})\neq 0,$$
and the left-hand side is equal to $x\wedge w^{x,y}\wedge y\wedge x\wedge w^{x,z}\wedge z$, which means that
$$w^{x,y}\wedge w^{x,z}\wedge x \wedge y \wedge z\neq 0.$$
So $w=w^{x,y}\wedge w^{x,z}$ works.
\end{proof}
	
\begin{proposition}\label{kp1}
Suppose $L$ is a quantale.
\begin{enumerate}
\item If $n,$ $n'\in L$ such that $n\lm n'^\downarrow$ and $n'$ is a $\m$-element of $L$, then $n$ is a $\m$-element of $L$.
			
\item If \( a_1\), \(a_2\), \(b_1\), \(b_2 \in L\)  such that \( a_1 \lm b_1^\downarrow \) and \( a_2 \lm b_2^\downarrow \), then \( (a_1 \wedge a_2) \lm (b_1 \wedge b_2)^\downarrow \).
			
\item\label{pp1} If $a$ is a $\m$-element in $L$, then $(a\wedge b)\lm b^\downarrow$, for every $b\in L$.
			
\item Suppose $L$ is a frame. Let $a\in L$ and $b$ be a pseudo-complement of  $a$ in $L$. If $c\in L$ is a $\m$-element in $L$, then $(b\vee c)\lm b^\uparrow$.   
\end{enumerate}
\end{proposition}
	
\begin{proof}
(1) Let $ k_1$ and $ k_2$  be nonzero elements in \(L\) such that 
$
k_1\wedge k_2  \neq 0.
$
We assume that \( n \wedge k_i \neq 0 \) for each \( i  \). Then $n'\wedge k_i\neq 0$ for each $k$.
Since \( n' \) is a $\m$-element in $L$, we have \(n'\wedge \left(  k_1\wedge k_2 \right)  \neq 0\). This implies that $0\neq  n'\wedge k_i\in n^\downarrow$ for each $i$ with $\bigwedge_{i=1}^2 (n'\wedge k_i)\neq 0.$ Since \( n\lm n'^\downarrow \), moreover, we have,
\[
0\neq n \wedge \left(\bigwedge_{i=1}^2 (n'\wedge k_i)\right) =n\wedge \left( k_1\wedge k_2  \right),
\]
which shows that \( n\) is a $\m$-element in $L$.
		
(2) Suppose that \( a_1 \lm b_1^\downarrow \) and \( a_2 \lm b_2^\downarrow \). Let \(n_1\) and \(n_2\) be nonzero elements in \( (b_1 \wedge b_2)^\downarrow \) with \( n_1\wedge n_2\neq 0\). Let \( (a_1 \wedge a_2) \wedge n_i \neq 0 \) for each \( i  \).
Define \( n_i^1 := n_i \wedge b_1 \) and \( n_i^2 := n_i \wedge b_2 \), so that \(\bigwedge_{i=1}^2 n_i^1 \neq 0\) and \( \bigwedge_{i=1}^2 n_i^2 \neq 0 \). Notice that \( n_i = n_i^1 \wedge n_i^2 \). Thus,
\[
\bigwedge_{i=1}^2 n_i = \bigwedge_{i=1}^2 (n_i^1 \wedge n_i^2).
\]
Since \((a_1 \wedge a_2) \wedge n_i \neq 0\) for each \( i \), we get
\[
0\neq(a_1 \wedge a_2) \wedge n_i = (a_1 \wedge a_2) \wedge (n_i^1 \wedge n_i^2) =(a_1 \wedge n_i^1) \wedge (a_2 \wedge n_i^2).
\]
This implies that \( a_1 \wedge n_i^1 \neq 0 \) and \( a_2 \wedge n_i^2 \neq 0 \). Since \( a_1 \lm b_1^\downarrow \) and \( a_2 \lm b_2^\downarrow \), we have
\[
a_1 \wedge \left( \bigwedge_{i=1}^2 n_i^1 \right) \neq 0\quad \text{and}\quad a_2 \wedge \left( \bigwedge_{i=1}^2 n_i^2 \right) \neq 0.
\]
Thus,
\[
(a_1 \wedge a_2) \wedge \left( \bigwedge_{i=1}^2 n_i \right) = (a_1 \wedge a_2) \wedge \left( \bigwedge_{i=1}^2 (n_i^1 \wedge n_i^2) \right) \neq 0.
\]
Therefore, \( (a_1 \wedge a_2) \lm (b_1 \wedge b_2)^\downarrow \).
		
(3) Let $x_1$ and $x_2$ be nonzero elements in $b^\downarrow$ such that $x_1\wedge x_2\neq 0$ and $(a\wedge b)\wedge x_k\neq 0$ for each $k$. This implies $a\wedge (b\wedge x_k)\neq 0$ for each $k$. Also, observe that $\bigwedge_{k=1}^2(b\wedge x_k)=b\wedge (x_1\wedge x_2)=x_1\wedge x_2\neq 0.$ Since $a$ is a $\m$-element in $L$, this implies that \[(a\wedge b)\wedge \left(x_1\wedge x_2\right)=a\wedge \left(\bigwedge_{k=1}^2(b\wedge x_k)\right)\neq 0,\]
proving the desired claim.
		
(4) We can assume $a\neq 0$ and $a\neq 1$, because otherwise $b=1$ or $b=0$ respectively, and the thesis holds. We can also assume that $c\not\leqslant b$, because otherwise $b\vee c=b\lm b^\uparrow$. Let $x_1>b$ and $x_2 >b$  such that  $ (x_1\wedge x_2)>b$ and $((b\vee c)\wedge x_k)>b$, for each $k$. The latter condition is equivalent to 
$b\vee (c\wedge x_k)>b$, which is equivalent to $c\wedge x_k\not\leqslant b$. We need to show that
$$(b\vee c)\wedge x_1\wedge x_2>b,$$ which is equivalent to $c\wedge x_1\wedge x_2\not\leqslant b$. Since $b$ is a pseudo-complement of $a$ in $L$, $c\wedge x_1\wedge x_2\not\leqslant b$ holds if and only if
$$a\wedge c\wedge x_1\wedge x_2\neq 0.$$
Now, consider the elements $a,x_1,x_2 \in L$. $a\wedge x_1\wedge x_2\neq 0$ because otherwise we would have $x_1\wedge x_2\leq b$, as $b$ is a pseudo-complement of $a$. Moreover $c\wedge x_k\neq 0$ for each $k$, because otherwise we would have $c\wedge x_k\leq b$. Finally $c\wedge a\neq 0$, because otherwise $c\leq b$, contradicting our assumption. As $c$ is a $\m$-element in $L$, we conclude that $c\wedge a\wedge x_1\wedge x_2\neq 0$, which proves the claim.
\end{proof}
	
\begin{remark}
An analogue of part (2) of Proposition \ref{kp1} does not hold for upsets. Consider the frame of power set $\mathcal{P}(X)$ of the set $X=\{1,2,3,4,5\}$. We have $\{1,2\}\lm \{1\}^\uparrow$ and $\{3,4\}\lm \{3\}^\uparrow$, by Remark \ref{remrem}. However, it is easy to see that $\{1,2\}\vee \{3,4\}=\{1,2,3,4\}$ is not a $\m$-element in $\{1,3\}^\uparrow={\{1\}\vee\{3\}}^\uparrow$. Moreover, it is also not a $\m$-element in $\mathcal{P}(X)$.
\end{remark}
	
\begin{proposition}\label{pcmu}
Let $L$ be a frame, $a\in L$, and 
$b$ a pseudo-complement of $a$ in $L$. If $c\in L$ is maximal relative to the
properties $a\leqslant c$ and $b\wedge c=0$, then $a\lm c^\downarrow$.
\end{proposition}

\begin{proof}
Let us first note that every frame is upper continuous, and hence, by \cite[Corollary 4.1]{Cal00}, $L$ is pseudo-complemented. Now, the existence of an element $c$ with the required properties follows from Zorn's lemma. So, what remains is to show that $a$ is a $\m$-element in $c^\downarrow$.
To show $a\lm c^\downarrow$, it is sufficient to show  $a\wedge x= 0$ implies that $x=0$, for any $x\in c^\downarrow$. Observe that
$$a\wedge(b\vee x) = (a\wedge c)\wedge (b\vee x) =a\wedge((x\vee b) \wedge c) = a\wedge(x\vee (b\wedge c)) =a\wedge(x\vee 0)=a\wedge x=0.
$$
Since $b$ is a pseudo-complement element of $a$, we must have $b\vee x=b$. In other words, $b\wedge x=x$. Hence$x=b\wedge x\leqslant b\wedge c=0$, implying that $x=0$.
\end{proof}
	
\begin{remark}
Unlike what happens for essential elements (see \cite[Lemma 4.3]{Cal00}), in Proposition \ref{pcmu}, we cannot conclude that $c$ is maximal with respect to the property $a\lm c^\downarrow$, not even for a frame. Here is a counterexample. Consider the frame of power set $\mathcal{P}(X)$ of the set $X=\{1,2,3\}$, and take $a=\{1\}$ and $b=\{2,3\}$ the complement of $a$. The maximal $c$ such that $a\leq c$ and $b\wedge c=0$ is $c=a$ itself, and indeed $a\lm a^\downarrow$. But we also have $a\lm \{1,2\}^\downarrow$, because $a$ is an atom also in $\{1,2\}^\downarrow$.
What makes the maximality of $c$ fail in Proposition \ref{pcmu} is that it is not true that if $a\lm c^\downarrow$ and $a\wedge y=0$ then $c\wedge y=0$, unlike what happens for essential elements (see \cite[Exercise 4.14]{Cal00}).
Indeed, in the frame considered above, $\{1\}\lm \{1,2\}^\downarrow$ and $\{1\}\wedge \{2,3\}=\emptyset$ but $\{1,2\}\wedge \{2,3\}=\{2\}\neq \emptyset$.
\end{remark}

We finish this section with a result on socles, and  revisit them in Section \ref{chmu}.	

\begin{proposition}
Suppose $L$ is a compactly generated quantale. Then the following hold.
\begin{enumerate}
\item $\bigwedge \left\{\m\!\!\text{-elements of}\; \,L\right\}\leqslant \mathrm{Soc}(L)$.
			
\item If $L$ does not have any proper $\m$-elements, then $\mathrm{Soc}(L)=1$.
\end{enumerate}
\end{proposition}
	
\begin{proof}
(1) Since by Proposition \ref{eim}, every essential element is a $\m$-element and since by Proposition \ref{bpmi}(\ref{atmu}), every atom is a $\m$-element, we can conclude that the set of all $\m$-elements contains the set of all atoms in $L$. By \cite[Theorem 5.1]{Cal00}, we know that \[\bigwedge \left
\{\text{essential elements of}\; \,L\right\}= \mathrm{Soc}(L).\]
From these, our desired inequality follows immediately.
		
(2) Suppose $L$ does not have any proper $\m$-elements. Then by Corollary \ref{css}, $L$ is complemented. Now by \cite[Theorem 6.7]{Cal00}, a complemented $L$ implies  $\mathrm{Soc}(L)=1$, and thus, we have the claim. 
\end{proof}
	
\section{$\m$-complements and $\m$-closedness}
\label{mucc}
	
In this section, we generalize essential complements to 
$\m$-complements and essential closedness to 
$\m$-closedness. Once again, many results known for essential complements and essential closedness extend naturally to this context, while others fail.
	
\begin{definition}
Let $L$ be a quantale. A \emph{$\m$-complement} of an element $x$ in $L$ is an element $y$ in $L$ such that $x\wedge y=0$ and $x\vee y$ is a  $\m$-element in $L$. 
An element $x$ in $L$ is called \emph{$\m$-closed} if $x\lm b^\downarrow$ implies $x=b$.
\end{definition}

Our terminology, ``$\m$-complement'' is an adaptation from ``essential complements'' of modules.  In case of $a\lm b^\uparrow$, note that the ``zero'' element is nothing but $b$, as it is the bottom element in the subframe $b^\uparrow$.
Notice that any complement is a $\m$-complement, as the top element is a $\m$-element.
	
\begin{remark}
$\m$-complements are not unique. Indeed, in the quantale of ideals of the ring $\mathds{Z}_{12}$, the elements  $3 \mathds{Z}_{12}$ and $4 \mathds{Z}_{12}$ are complements of each other, as $3 \mathds{Z}_{12}\wedge 4 \mathds{Z}_{12}=0$ and $3 \mathds{Z}_{12}\vee 4 \mathds{Z}_{12}=\mathds{Z}_{12}$. Consider now $2 \mathds{Z}_{12}$ and the intersections $3 \mathds{Z}_{12}\wedge 2 \mathds{Z}_{12}=6\mathds{Z}_{12}$ and $4 \mathds{Z}_{12}\wedge 2 \mathds{Z}_{12}=4 \mathds{Z}_{12}$. It is easy to see that $2 \mathds{Z}_{12}$ is a $\m$-element. Moreover, $6 \mathds{Z}_{12}\wedge 4 \mathds{Z}_{12} = 0$ and $6 \mathds{Z}_{12}\vee 4\mathds{Z}_{12}=2 \mathds{Z}_{12}$. So, both $3 \mathds{Z}_{12}$ and $6 \mathds{Z}_{12}$ are $\m$-complements of $4 \mathds{Z}_{12}$.
\end{remark}
	
\begin{proposition}
Suppose $L$ is a distributive quantale. If $a$, $n\in L$ are such that  $a\leqslant n$ and $b$ is a $\m$-complement of $a$ in $L$, then $b \wedge n$ is a $\m$-complement of $a$ in $n^\downarrow$.
\end{proposition}
	
\begin{proof}
Since $a \wedge b = 0$ and $b \wedge n \leqslant b$, it follows that
\(a \wedge (b \wedge n) \leqslant a \wedge b = 0.
\)
To show that $a \vee (b \wedge n)\lm n^\downarrow$,
let $k_1$ and $k_2$ be  nonzero elements in $n^\downarrow$ with
\(
\ k_1\wedge k_2 \neq 0\) and \((a \vee (b \wedge n)) \wedge k_i \neq 0 \) for each $i$.
Since $a \vee (b \wedge n) \leqslant a \vee b$ and $k_i \leqslant n \in L$, we have
\(
0\neq (a \vee (b \wedge n)) \wedge k_i = (a \vee b) \wedge k_i\) for each $i$, by applying distributivity. 
Thus, we can rewrite
\[
(a \vee (b \wedge n)) \wedge \left( k_1\wedge k_2  \right) = (a \vee b) \wedge \left( k_1\wedge k_2  \right).
\]
Since  
\(
k_1\wedge k_2  \neq 0\) and \((a \vee b) \wedge k_i \neq 0
\) for each $i$, and $a \vee b$ is a $\m$-element, we obtain
\(
(a \vee b) \wedge \left( k_1\wedge k_2  \right) \neq 0.\)
Therefore,
\[
(a \vee (b \wedge n)) \wedge \left( k_1\wedge k_2  \right) \neq 0,
\]
proving the claim.
\end{proof}

It can be shown  that $a\vee a^\perp$ is an essential element in $L$. The above result in fact extends this to $\m$-elements and shows more.
		
\begin{theorem}\label{mcm}
If \(L\) is a modular quantale and  \(x\), \(y\in L\)  such that \(x \wedge y= 0\), then \(y\) has a \( \m \)-complement containing \(x\).
\end{theorem}
	
\begin{proof}
Consider the set \[\mathcal{A}:= \left
\{ y'\in L\mid y \wedge y' = 0,\; x \leqslant y'\right\}.\]
The set \(\mathcal{A}\) is nonempty because \(x \in \mathcal{A}\). By Zorn's Lemma, \(\mathcal{A}\) has a maximal element, say \(y'\). 
We claim that \(y \vee y'\) is a $\m$-element in $L$. If not, then there exist  nonzero elements \(x_1\) and \(x_2\) in \(L\) such that \(x_1\wedge x_2 \neq 0\) and \((y\vee y') \wedge x_k \neq 0\), for each \(k\), but \[(y \vee y') \wedge \left(x_1\wedge x_2 \right) = 0.\] Then by \cite[Lemma 6.3]{Cal00}, we have $y\wedge \left(y'\vee (x_1\wedge x_2 )\right) = 0,$ contradicting the maximality of \(y'\). 
\end{proof}
	
\begin{corollary}\label{emms}
In a modular quantale $L$, the following hold.
\begin{enumerate}
\item Every element 
in $L$ has a
\( \m \)-complement.
	
\item\label{css} If $L$ has no proper $\m$-elements, then $L$ is complemented.
	
\item Every pseudo-complement in $L$  is a $\m$-complement.
\end{enumerate}
\end{corollary}
	
\begin{proof}
(1)		Take $x=0$ in Theorem \ref{mcm}.

(2) If a quantale $L$ does not have any proper $\m$-elements, then by (1), for every  $x\in L$, there exists an  $y\in L$ such that $1=x\vee y$, and hence, $L$ is complemented.

(3) Let $x$ be a pseudo-complement of $y$ in $L$. Then $x\wedge y=0$ and by Theorem \ref{mcm} $y$ has a $\m$-complement $x'$ containing $x$. In particular we have $x'\wedge y=0$ with $x'\geqslant x$, and by maximality of $x$ we conclude that $x'=x$. So $x$ is a $\m$-complement of $y$.
\end{proof}

Since every essential element is a $\m$-element, obviously, every $\m$-closed element in $L$ is essentially closed in $L$. 		
	
\begin{proposition}\label{proppseudocompl}
Let $L$ be a modular quantale and $a\in L$. If $a$ is $\m$-closed in $L$, then $a$ is a pseudo-complement of an element in $L$.
\end{proposition}

\begin{proof}
If $a$ is $\m$-closed in $L$, then $a$ is essentially closed in $L$. By \cite[Corollary 4.3]{Cal00}, we conclude that $a$ is a pseudo-complement.
\end{proof}
	
\begin{remark}\label{remexmuclosed}
Unlike what happens for essential elements (see \cite[Corollary 4.3]{Cal00}), the converse of Proposition \ref{proppseudocompl} does not hold, not even for a frame. Indeed, consider the frame of power set $\mathcal{P}(X)$ of the set $X=\{1,2,3\}$. We have that $\{1\}$ is a pseudo-complement of an element (actually, a complement), but $\{1\}$ is not $\m$-closed as $\{1\}\lm\{1,2\}^\downarrow$, since it is an atom there.
\end{remark}
	
\begin{proposition}\label{theormuclosed}
Suppose $L$ is a frame and $a$ and $b$ are nonzero elements in $L$.  If $a$ is $\m$-closed in $L$, then whenever $a\leqslant b$ and $b$ is a $\m$-element in $L$, we have that $b\lm a^\uparrow$.
\end{proposition}
	
\begin{proof}
Suppose that $a$ is $\m$-closed in $L$. Assume that $a\leqslant b$ and that $b$ is a $\m$-element in $L$; we need to show that $b\lm a^\uparrow$. Let $a<x_1$ and  $a<x_2$ such that $a<(x_1\wedge x_2)$ and $a<b\wedge x_k$ for each $k$. It is clear that $b\wedge \left( x_1\wedge x_2 \right)\geqslant a$. Therefore, it is sufficient to show that if $b\wedge \left( x_1\wedge x_2 \right)= a$, then $a=x_1\wedge x_2$. By Proposition  \ref{kp1}(\ref{pp1}), we have \[b\wedge \left( x_1\wedge x_2 \right) \lm \left( x_1\wedge x_2 \right)^\downarrow.\] Hence $a\lm \left( x_1\wedge x_2 \right)^\downarrow$. Since $a$ is $\m$-closed in $L$, we must have $a=x_1\wedge x_2$.
\end{proof}
	
\begin{remark}
Unlike what happens for essential elements (see \cite[Exercise 4.12]{Cal00}, the converse of Proposition \ref{theormuclosed} does not hold. Here is a counterexample. Consider the frame of power set $\mathcal{P}(X)$ of the set $X=\{1,2,3\}$, and take $a=\{1\}$. The $\m$-elements of $\mathcal{P}(X)$ are precisely $\emptyset,\{1\},\{2\},\{3\},X$. So, whenever $\{1\}\leq b$ and $b$ is a $\m$-element in $\mathcal{P}(X)$, we have $b\lm \{1\}^\uparrow$. Indeed, the only possible elements for $b$ are $\{1\}$ and $X$, and $X$ is a $\m$-element in $\{1\}^\uparrow$ because it is its top element. However $a=\{1\}$ is not $\m$-closed, by Remark \ref{remexmuclosed}.
\end{remark}
	
\section{Characterizations of $\m$-elements}
\label{chmu}	

In this section, we first establish a complete characterization of 
$\m$-elements in certain distinguished classes of quantales, including the quantales of ideals in a quotient ring of a principal ideal domain (PID) and the frames of open subsets of a topological space. Building on these results, we then provide a complete characterization of 
$\m$-elements in every modular quantale.
	
We begin by proving a characterization of $\m$-elements in the quantale of ideals in  $\mathds{Z}_n$.
	
\begin{theorem}\label{charzn}
Let $n\in\mathds{N}$, $n\geqslant 2$, and consider its unique factorization $n=p_1^{m_1}\cdot \ldots \cdot p_k^{m_k}$. In the quantale of ideals of the ring $\mathds{Z}_n$, an element 
$$x:=(l)=(\operatorname{gcd}(l,n))=\left(p_1^{m'_1}\cdot \ldots \cdot p_k^{m'_k}\right)$$
is a $\m$-element if and only if there are not $i,j,s \in \{1,\ldots, k\}$ such that $i\neq j$, $m'_i < m_i$, $m'_j < m_j$ and $m'_s=m_s$. 
In particular, $x$ is essential if and only if for every $i$ we have $m'_i < m_i$.
\end{theorem}
	
\begin{proof}
%two cases: not i,j such that.. then mu-el ; i,j then not mu-el
We first prove that $x$ is essential if and only if for every $i$ we have $m'_i < m_i$. If $m'_i < m_i$, given an ideal $y\neq (0)$ with $y=(p_1^{m''_1}\cdot \ldots \cdot p_k^{m''_k})$, we obtain 
$$x \wedge y= \left(p_1^{\operatorname{max}\{m'_1,m''_1\}}\cdot \ldots \cdot p_k^{\operatorname{max}\{m'_k,m''_k\}}\right).$$
Since $y\neq (0)$, there exists $j\in \{1,\ldots, k\}$ such that $m''_j < m_j$ and so $x\wedge y\neq (0)$ since $p_j$ appears with exponent $\operatorname{max}\{m'_j,m''_j\}< m_j$ in the unique factorization of the generator of $x \wedge y$. Conversely, if $x$ is essential, then for every $i\in\{1,\ldots, k\}$ we must have $$x \wedge \left(p_1^{m_1}\cdot \ldots \cdot p_i^{m_i-1} \cdot \ldots \cdot p_k^{m_k}\right) \neq (0),$$ and thus $\operatorname{max}\{m_i-1, m'_i\} < m_i$. So it must be $m'_i < m_i$. 
		
Let us now consider the case in which $x$ is not essential, and thus there exists $s\in \{1,\ldots,k\}$ such that $m'_s=m_s$. Suppose there are not $i,j\in \{1,\ldots, k\}$ such that $i\neq j$, $m'_i < m_i$, $m'_j < m_j$. 
If $m'_j=m_j$ for every $j\neq i$ and $m'_i=m_i -1$, then $x$ is an atom and thus a $\m$-element. Indeed, if an ideal $$y=\left(p_1^{m''_1}\cdot \ldots \cdot p_k^{m''_k}\right)$$ is  contained in $x$ we must have $m_j \leqslant m''_j$ for every $j\neq i$ and $m_i-1 \leqslant m''_i$ and thus either $x=y$, or $y=(0)$. Otherwise, there must exist $i\in \{1, \ldots, k\}$ such that $m'_j=m_j$ for every $j\neq i$ and $m'_i < m_i-1$. Let $y$, $z$ be ideals of $\mathds{Z}_n$ such that $y\wedge z \neq (0)$, $x \wedge y \neq (0)$ and $x \wedge z \neq (0)$. Since $x \wedge y \neq (0)$, the prime $p_i$ must appear in the factorization of the generator of $y$ with an exponent strictly smaller than $m_i$. And analogously for the generator of $z$. This implies that $x \wedge y\wedge z \neq (0)$ since the exponent of the prime $p_i$ in the unique factorization of the generator of $x \wedge y\wedge z$ is strictly less than $m_i$. This proves that $x$ is a $\m$-element. 
		
Suppose now that there exists $i,j\in \{1,\ldots, k\}$ such that $i\neq j$, $m'_i < m_i$, $m'_j < m_j$ and consider the following ideals: 
$$y=\left(p_1^{m_1}\cdot \ldots \cdot p_i^{m_i} \cdot \ldots \cdot p_j^{m'_j} \cdot \ldots \cdot p^0_s\cdot \ldots \cdot p_k^{m_k}\right) \quad \text{ and } \quad z=\left(p_1^{m_1}\cdot \ldots \cdot p_i^{m'_i} \cdot \ldots \cdot p_j^{m_j} \cdot \ldots \cdot p^0_s\cdot \ldots \cdot p_k^{m_k}\right).$$
Then, since $m_s>0$, we have 
$$y\wedge z=\left(p_1^{m_1}\cdot \ldots \cdot p_i^{m_i} \cdot \ldots \cdot p_j^{m_j} \cdot \ldots \cdot p^0_s\cdot \ldots \cdot p_k^{m_k}\right) \neq (0).$$
Moreover, since since $m'_j < m_j$, we have $$x \wedge y=\left(p_1^{m_1}\cdot \ldots \cdot p_i^{m_i} \cdot \ldots \cdot p_j^{m'_j} \cdot \ldots \cdot p_s^{m_s}\cdot \ldots \cdot p_k^{m_k}\right)\neq (0)$$  and analogously $x \cap z\neq (0)$ because $m'_i < m_i$. But clearly $x\wedge y\wedge z=(0)$. So, $x$ is not a $\m$-element.
\end{proof}
	
The previous characterization can be proved more generally for the quantale of ideals of a quotient ring of a PID, and the proof is completely analogous to the one of Theorem \ref{charzn}.
	
\begin{theorem}\label{charquopid}
Let $R$ be a PID and let $I:=(a)\subseteq R$ be an ideal of $R$. Consider the unique factorization $a=p_1^{m_1}\cdot \ldots \cdot p_k^{m_k}$ of $a\in R$ as product of prime elements. In the quantale of ideals of the ring $R/I$, an element 
$$x:=(l)=(\operatorname{gcd}(l,n))=\left(p_1^{m'_1}\cdot \ldots \cdot p_k^{m'_k}\right)$$
is a $\m$-element if and only if there are not $i,j,s \in \{1,\ldots, k\}$ such that $i\neq j$, $m'_i < m_i$, $m'_j < m_j$ and $m'_s=m_s$. 
In particular, $x$ is essential if and only if for every $i$ we have $m'_i < m_i$.
\end{theorem}
	
The following example shows that a product of $\m$-elements is not always a $\m$-element.
\begin{example}\label{exprodmuel}
Consider the quantale of ideals of $\mathds{Z}_{900}$. Since $900=2^2\cdot 3^2 \cdot 5^2$, by Theorem \ref{charzn}, the ideals $(6)=(2\cdot 3)$ and $(10)=(2\cdot 5)$ are $\m$-elements. However, again by Theorem \ref{charzn}, $(6)\cdot(10)=(60)=(2^2\cdot 3 \cdot 5)$ is not a $\m$-element. 
\end{example}
	
\begin{example}
Consider the quantale $L$ of ideals of $\mathds{Z}_{180}$. We notice that the ideal $(5)$, despite being a maximal element of $L$, is not a $\m$-element by Theorem \ref{charzn}. Moreover, we observe that also the radical of $L$, which is the intersection of the maximal elements and thus $(2)\cdot (3)\cdot (5)=(2\cdot 3 \cdot 5)=(30)$, is not a $\m$-element, again thanks to Theorem \ref{charzn}.
\end{example}
	
\begin{example}
Consider the quantale of ideals of the ring $R=\mathds{R}[x]/(x^2(x+1)(x+2))$. Using Theorem \ref{charquopid} we can easily find out which ideals of $R$ are $\m$-elements by looking at their generator. We conclude that $(0)$, $ (x)$, $(x^2(x+1))$, $(x^2(x+2))$, $((x+1)(x+2))$, $(x(x+1)(x+2))$ and $R$ are $\m$-elements, while $(x^2)$, $(x+1)$, $(x+2)$, $(x(x+1))$ and $(x(x+2))$ are not $\m$-elements.
\end{example}
	
\begin{theorem}\label{theorcharactopens}
Let $X$ be a topological space and let $\operatorname{Open}(X)$ be the frame of open subsets of $X$. An open $U\in\operatorname{Open}(X)$ is a $\m$-element if and only if $U$ is dense or irreducible.
In particular, $U$ is essential in $\operatorname{Open}(X)$ if and only if $U$ is dense in $X$.
\end{theorem}
	
\begin{proof}
We first prove that $U$ is essential if and only if it is dense in $X$. If $U$ is essential, then given $V\in \operatorname{Open}(X)$ with $V\neq \emptyset$ we must have $U \wedge V \neq \emptyset$ and thus $U$ is dense. Conversely, if $U$ is a dense subspace of $X$, it must intersect all non-empty open subsets of $X$ and thus $U$ is essential.
		
We now prove that if $U$ is irreducible then $U$ is a $\m$-element.  Let $V$, $W\in \operatorname{Open}(X)$ such that $U\wedge V\neq \emptyset$, $U\wedge W\neq \emptyset$ and $V\wedge W\neq \emptyset$. Then $U\wedge V \wedge W \neq \emptyset$, because $$U\wedge V \wedge W=(U \wedge V) \wedge (U \wedge W)$$ and $U \wedge V$ and $U \wedge W$ are non-empty open subsets of $U$ which is irreducible. Hence, $U$ is a $\m$-element.
We have thus proved that if $U$ is either dense or irreducible it is a $\m$-element.
		
For the converse, assume that $U$ is a $\m$-element and that $U$ is not dense. We prove that $U$ is irreducible. Let $V$, $W \leqslant U$ open subsets of $U$ such that $V\neq \emptyset\neq W$; we show that $V\wedge W\neq \emptyset$. Notice that, since $U$ is open, we have that $V$ and $W$ are open subsets of $X$. As $U$ is not dense, there exists $Z\in \operatorname{Open}(X)$ such that $Z \neq \emptyset$ and $U\wedge Z= \emptyset$. Consider now $V\vee Z$, $W \vee Z\in\operatorname{Open}(X)$. Then $U\wedge (V \vee Z)=V \neq \emptyset$, $U\wedge (W \vee Z)=W \neq \emptyset$, and $(V \vee Z) \wedge (W \vee Z)\geqslant Z\neq \emptyset$. Since $U$ is a $\m$-element, this implies that $U \wedge (V \vee Z) \wedge (W \vee Z) \neq \emptyset$. And
\[U \wedge (V \vee Z) \wedge (W \vee Z)=U \wedge (V \vee Z) \wedge U \wedge (W \vee Z)=V\wedge W. \qedhere\]
\end{proof}
	
In particular, we obtain a characterization of $\m$-elements in the frame of power set of a non-empty set $X$.
	
\begin{corollary}
Let $X$ be a set. The $\m$-elements of the quantale $\mathcal{P}(X)$ of power set of $X$ are exactly $\emptyset$, $X$ and the singletons.
\end{corollary}
	
\begin{proof}
Consider $X$ as a discrete topological space. Then, by Theorem \ref{theorcharactopens}, the $\m$-elements of $\mathcal{P}(X)=\operatorname{Open}(X)$ are exactly the dense subsets of $X$ and the irreducible ones. But the only dense subspace of $X$ for the discrete topology is $X$ itself and the irreducible ones are exactly $\emptyset$ (in a trivial way) and the singletons.
\end{proof}
	
Theorem \ref{theorcharactopens} opens the way towards a complete characterization of $\m$-elements in any modular quantale. With a suitable notion of \emph{irreducible element} in a quantale, we will generalize Theorem \ref{theorcharactopens} to the much wider context of modular quantales.
	
\begin{definition}\label{defirreducible}
A quantale $L$ is \emph{irreducible} if given any $y$, $z \in L$ such that $y \wedge z=0$, we have that either $y=0$ or $z=0$. An element $x\in L$ is \emph{irreducible} if the quantale $x^{\downarrow}$ is irreducible. Equivalently, $x$ is irreducible if given any $y$, $z \in L$ with $y\leqslant x$ and $z \leqslant x$ such that $y \wedge z=0$ we have that either $y=0$ or $z=0$.
\end{definition}
	
\begin{remark}
Definition \ref{defirreducible} generalizes the definition of irreducible element given in \cite{D04} for frames. In  particular, for the  frame of open subsets of a topological spaces, this definition recovers the notion of irreducible topological space involved in Theorem \ref{theorcharactopens}.
As observed in \cite{D04}, this notion of irreducible element is different from the usual notion of \emph{meet-irreducible} element in a lattice. Indeed, our notion involves a primeness (with respect to meets) condition for the bottom element of the downset of $x\in L$ rather than a primeness condition for  the element $x$ itself.
\end{remark}
	
The following result generalizes Theorem \ref{theorcharactopens} to the case of a modular quantale.
	
\begin{theorem}\label{theorcharact}
Let $L$ be a modular quantale. An element $x$ in $L$ is a $\m$-element if and only if $x$ is either essential or irreducible.
\end{theorem}

\begin{proof}
If $x$ is essential, then $x$ is clearly a $\m$-element in $L$. Now, let $x$ be an irreducible element in $L$ and let $y$, $z\in L$ be such that $y\wedge z \neq 0$, $x \wedge y \neq 0$, and $x \wedge z \neq 0$. We prove that $c\wedge z \neq 0$. Notice that $x \wedge y$ and $x \wedge z$ are non-zero elements of $x^{\downarrow}$, and hence the irreducibility of $x$ implies that $(x \wedge y) \wedge (x \wedge z)\neq 0$. But $x \wedge y \wedge z=(x \wedge y) \wedge (x \wedge z),$ and hence, we conclude that $x$ is a $\m$-element.
		
Conversely, suppose that $x$ is a $\m$-element, but not essential. This means that there exists $w\in L$ such that $w\neq 0$ and $x \wedge w=0$. We show that $x$ is irreducible. Let $y$, $z\in L$ such that $y\neq 0 \neq z$, $y \leqslant x$ and $z \leqslant x$. We prove that $y \wedge z \neq 0$. Consider $y'=y \vee w$ and $z'=z \vee w$. Then, we have $y' \wedge z' \geqslant w \neq 0$. Moreover, $$x \wedge y'= x \wedge (y \vee w)=y \wedge (x \wedge w),$$ thanks to the modularity of $L$, and hence $x \wedge y'=y \neq 0$ since $x \wedge w=0$. Analogously, $x \wedge z'=z \neq 0$. Since $x$ is a $\m$-element, this implies $x \wedge y' \wedge z' \neq 0$. But $$x \wedge y' \wedge z'= (x \wedge y') \wedge (x \wedge z')= y \wedge z.$$ This shows that $x$ is irreducible.
\end{proof}
	
\begin{remark}
Theorem \ref{theorcharact} explains the  reason why atoms are always $\m$-elements. Interestingly, the natural generalization of essential elements given by shifting from $n=1$ to consider all $n\geqslant 2$ captures and simultaneously generalizes also irreducible elements.
\end{remark}
	
\begin{remark}
Theorem \ref{theorcharact} offers a different way to prove the explicit characterization of $\m$-elements given in Theorem \ref{charquopid} for the quantale of ideals of a ring $R/I$ with $R$ a PID and $I:=(p_1^{m_1}\cdot \ldots \cdot p_k^{m_k})$ an ideal in $R$. Notice that if $$x:=\left(p_1^{m'_1}\cdot \ldots \cdot p_k^{m'_k}\right),$$ and there exist $i,$ $j,$ $s \in \{1,\ldots, k\}$ such that $i\neq j$, $m'_i < m_i$, $m'_j < m_j$, and $m'_s=m_s$, then $x$ is not irreducible because the ideals 
\[\left(p_1^{m_1}\cdot \ldots \cdot p_i^{m_i} \cdot \ldots \cdot p_j^{m'_j} \cdot \ldots \cdot p_k^{m_k}\right)\qquad \text{and} \qquad \left(p_1^{m_1}\cdot \ldots \cdot p_i^{m'_i} \cdot \ldots \cdot p_j^{m_j} \cdot \ldots p_k^{m_k}\right)\] are contained in $x$ and are non-zero, but their meet is $(0)$.
\end{remark}
	
\begin{remark}
Thanks to Theorem \ref{theorcharact}, we can understand when the socle $\mathrm{Soc}(L)$ of a modular quantale $L$ is a $\m$-element. If $L$ has no atoms, then $\mathrm{Soc}(L)=0$ is trivially a $\m$-element. If $L$ has only one atom $a$, then $\mathrm{Soc}(L)=a$ is irreducible, and thus a $\m$-element. If $L$ has more than one atom, then $\mathrm{Soc}(L)$ is not irreducible, as the meet of any two distinct atoms $a$ and $b$ is 0. Indeed, $a\wedge b\leq a$, whence $a\wedge b=0$ or $a\wedge b=a$. But if $a\wedge b=a$ then $a\leq b$, whence $a=b$. So, if $L$ has more than one atom, by Theorem \ref{theorcharact}, $\mathrm{Soc}(L)$ is a $\m$-element precisely when it is essential. This is the case, for example, when $L$ is atomic (see
\cite[Proposition 5.1]{Cal00}).
\end{remark}
	
We can, in fact, generalize the above remark to the context of independent subsets of a modular quantale.
	
\begin{remark}
Let $M$ be an independent subset (see Section \ref{bck} for definition) in a modular quantale $L$.  It is shown in \cite{Cal00} that $\bigvee M$ is an essential element if and only if $M$ is maximal independent. Notice that if $M$ contains more than one element, then $\bigvee M$ cannot be irreducible, by definition of independent subset. So, by Theorem \ref{theorcharact}, $\bigvee M$ is a $\m$-element if and only if $M$ is maximal independent or $M$ contains only one element which is irreducible.
\end{remark}
	
\begin{theorem}\label{charmuclosed}
Let $L$ be a modular quantale and let $x\in L$ with $x\neq 1$. The following are equivalent:
\begin{enumerate}
\item $x$ is $\m$-closed;

\item $x$ is essentially closed and not irreducible.
\end{enumerate}
Moreover, $x=1$ is both $\m$-closed and essentially closed.
\end{theorem}

\begin{proof}
Let $x\neq 1$ be $\m$-closed. Since every essential element is a $\m$-element, $x$ is clearly essentially closed. Moreover, $x$ is not irreducible. Indeed, if $x$ is irreducible in $L$, by Theorem \ref{theorcharact}, $x$ is a $\m$-element and so $x \lm 1^\downarrow$. Since $x\neq 1$, this implies that $x$ is not $\m$-closed which gives a contradiction. Hence, $x$ is not irreducible in $L$.
		
To prove the converse, assume that $x\neq 1$ is essentially closed and not irreducible. Let then $a\in L$ be such that $x \lm a^\downarrow$. We show that $a=x$. Since $a^\downarrow$ is a modular quantale, by Theorem \ref{theorcharact}, we have that $x$ is either irreducible in $a^\downarrow$ or it is essential in $a^\downarrow$. If $x$ is irreducible in $a^\downarrow$, then $x$ is irreducible in $L$ as well, which gives a contradiction. If $x$ is essential in $a^\downarrow$ then $x=a$ because $x$ is essentially closed. Hence $x$ is $\m$-closed. 
\end{proof}
	
Thanks to our characterization of $\m$-closedness in a modular quantale, we can prove the following, which is known for essential closedness.
	
\begin{corollary}
Suppose $L$ is a modular quantale and $a$, $b\in L$. If $a\leqslant b$, $a$ is $\m$-closed in $b^\downarrow$, and $b$ is $\m$-closed in $L$, then $a$ is $\m$-closed in $L$.
\end{corollary}
	
\begin{proof}
By Theorem \ref{charmuclosed}, $a$ is essentially closed in $b^\downarrow$ and not irreducible in $b^\downarrow$. By  \cite[Exercise 4.16]{Cal00}, $a$ essentially closed in $b^\downarrow$ implies that $a$ essentially closed in $L$. Moreover, by the definition of irreducibility, $a$ not irreducible in $b^\downarrow$ implies that $a$ is not irreducible in $L$. Hence, by Theorem \ref{charmuclosed}, we conclude that $a$ is $\m$-closed in $L$.
\end{proof}
	
\section*{Acknowledgements}
The second author thanks Themba Dube for some fruitful discussions.
	
\section*{Declarations}

\textbf{Funding.} Research did not receive any funding.

\textbf{Conflict of interest.} The authors declare that they have no conflict of interest.

\textbf{Human and animals rights.} Research did not involve human participants and/or animals.


\begin{thebibliography}{99}
%\bibitem{Arm68} E. P. Armendariz, Direct and subdirect sums of simple rings with unit, \textit{Amer. Math. Monthly}, \textbf{75} (1968), 746--748.

\bibitem{Aza97}
F. Azarpanah, Intersection of essential ideals in $C(X)$, \textit{Proc. Amer. Math. Soc.}, \textbf{125} (1997), 2149--2154

\bibitem{Aza95}
\bysame, Essential ideals in $C(X)$, \textit{Period. Math. Hungar.}, \textbf{31} (1995), 105--112.
		
\bibitem{AE72}
E. M. Alfsen and E. G. Effros, Structure in real Banach spaces. Part I, \textit{Ann. Math.}, \textbf{96}(1) (1972), 98--128.		

\bibitem{BD22}
P. Bhattacharjee and T. Dube,  On fraction-dense algebraic frames. \textit{Algebra Universalis}, \textbf{83}(6) (2022).
		
\bibitem{BG25}
D. P. Blecher and A. Goswami, On $\mathcal{M}$-extensions of modules, in preparation.
		
\bibitem{Cal00}
G. C\v{a}lug\v{a}reanu, \textit{Lattice concepts of module theory}, Kluwer Academic Publishers, Dordrecht, 2000.
		
\bibitem{Dil62}
R. P. Dilworth,  Abstract commutative ideal theory, \textit{Pacific J. Math.}, \textbf{12} (1962), 481--498.
		
\bibitem{D04}
T. Dube, Irreducibility in pointfree topology, \textit{Quaestiones Mathematicae}, \textbf{27}(3), (2004),  231--241.
		
		
\bibitem{DB23}
T. Dube and S. Blose, Algebraic frames in which dense elements
are above dense compact elements, \textit{Algebra Universalis}, \textbf{84}(3), (2023), 28 pp.
		
\bibitem{EN75}
J. Eastham and W. Nemitz, Density and closure in implicative semi-lattices, \textit{Algebra Universalis}, \textbf{5} (1975), 1--7.
		
\bibitem{Gre05}
D. J. Green, The essential ideal is a Cohen-Macaulay module, \textit{Proc. Amer. Math. Soc.}, \textbf{133}(11) (2005), 3191--3197.
		
\bibitem{GD25} A. Goswami and T. Dube, On $z$-elements of multiplicative lattices, \textit{Algebra Universalis}, \textbf{86}(4)
		(2025), 23 pp.
		
\bibitem{HM07}
A. W. Hager and J. Marti\'{n}ez,  Patch-generated frames and projectable hulls, \textit{Appl.
Categ. Struct.}, \textbf{15}, (2007), 49--80.

\bibitem{KRB22}
A. P. P. Kumar, M. S. Rao, and K. S. Babu, Filters of distributive lattices genrated by dense elements, \textit{Palest. J. Math.}, \textbf{11}(2) (2022), 45--54.

\bibitem{Mom10}
E. Momtahan, Essential ideals in rings of measurable functions, \textit{Comm. Alg.}, \textbf{38} (2010), 4739--4746.
		
%\bibitem{Mul86} C. J. Mulvey, \&, \textit{Rend. Circ. Mat. Palermo (2) Suppl.} No. \textbf{12} (1986), 99--104. 
		
%\bibitem{PP16} J. Picado and A. Pultr, New aspects of subfitness in frames and spaces, \textit{Appl. Categ. Struct.}, \textbf{24} (2016), 703--714.
		
\bibitem{PP12} J. Picado and A. Pultr, \textit{Frames and locales: topology without points}, Birkh\"{a}user, Basel, 2012.
		
\bibitem{OJ81}
D. M. Olson and T. L. Jenkins, Upper radicals and essential ideals, \textit{J. Aust. Math. Soc.}, \textbf{30} (1981), 385--389. 
		
\bibitem{Ros90}  K. I. Rosenthal, \textit{Quantales and their applications}, Pitman research notes in Math. No. 234, Longman, Scientific 
and  Technical, 1990.
		
\bibitem{Row88} L. H. Rowen, \textit{Ring theory}, vol. I, Academic Press, Inc., 1988.
		
\bibitem{Tah15}
A. Taherifar, Essential ideals in subrings of $C(X)$ that contain $C^*\!(X)$, \textit{Filomat}, \textbf{29}(7) (2015), 1631--1637.
		
\bibitem{WD39}  M. Ward and R. P. Dilworth, Residuated lattices, \textit{Trans. Amer. Math. Soc.}, \textbf{45} (1939), 335--354.
\end{thebibliography}
\end{document}